\documentclass[onecolumn]{autart}

\usepackage[T1]{fontenc}
\usepackage[cp1252]{inputenc}
\usepackage[english]{babel}
\usepackage{amsmath,amssymb}
\usepackage{lmodern}
\usepackage{amsfonts}
\usepackage{url}
\usepackage[tight,footnotesize]{subfigure}
\usepackage{psfrag}
\usepackage{graphicx}
\usepackage{epstopdf}
\usepackage{algorithmic,cite}
\usepackage{algorithm}

\usepackage[bookmarks=true,colorlinks,citecolor=black]{hyperref}

\usepackage[colorinlistoftodos, textwidth=3cm, shadow, disable]{todonotes}

\graphicspath{{figures/}{figures/inkscape/}}
\DeclareGraphicsExtensions{.eps,.pdf,.png}

\newtheorem{lemma}[thm]{Lemma}

\newtheorem{proposition}[thm]{Proposition}
\newtheorem{assumption}[thm]{Assumption}

\newtheorem{definition}[thm]{Definition}

\newenvironment{lemma*}[2][Lemma]{\begin{trivlist}
\item[\hskip \labelsep {\bfseries #1}\hskip \labelsep {\bfseries #2}]\itshape }{ \end{trivlist}}

\begin{document}
\begin{frontmatter}
\title{A Secure Control Framework for Resource-Limited Adversaries}

\thanks[footnoteinfo]{This paper was not presented at any IFAC
meeting. Corresponding author Andr\'{e} Teixeira. Tel. +46-73-429 78 31.
Fax +46-8-790 73 29.}

\author{Andr\'{e} Teixeira$^\star$}\ead{\texttt{andretei@kth.se}},    % Add the
\author{Iman Shames$^\dagger$}\ead{\texttt{iman.shames@unimelb.edu.au}},               % e-mail address
\author{Henrik Sandberg$^\star$}\ead{\texttt{hsan@kth.se}},  % (ead) as shown
\author{Karl H. Johansson$^\star$}\ead{\texttt{kallej@kth.se}}

\address{$^\star$ACCESS Linnaeus Centre, KTH Royal Institute of Technology, Electrical Engineering, Stockholm, Sweden }
\address{$^\dagger$Department of  Electrical and Electronic Engineering, University of Melbourne, Australia }


\begin{abstract}
Cyber-secure networked control is modeled, analyzed, and experimentally illustrated in this paper. An attack space defined by the adversary's system knowledge, disclosure, and disruption resources is introduced. Adversaries constrained by these resources are modeled for a networked control system architecture. It is shown that attack scenarios corresponding to denial-of-service, replay, zero-dynamics, and bias injection attacks can be analyzed using this framework. Furthermore, the attack policy for each scenario is described and the attack's impact is characterized using the concept of safe sets. An experimental setup based on a quadruple-tank process controlled over a wireless network is used to illustrate the attack scenarios, their consequences, and potential counter-measures.
\end{abstract}

\end{frontmatter}

\section{Introduction}
\label{sec:intro}

Safe and reliable operation of infrastructures is of major societal importance. These systems need to be engineered in such a way so that they can be continuously monitored, coordinated, and controlled despite a variety of potential system disturbances. Given the strict operating requirements and system complexity, such systems are operated through IT infrastructures enabling the timely data flow between digital controllers, sensors, and actuators. However, the use of communication networks and heterogeneous IT components has made these cyber-physical systems vulnerable to cyber threats. One such example are the industrial systems and critical infrastructures operated through Supervisory Control and Data Acquisition (SCADA) systems. The measurement and control data in these systems are commonly transmitted through unprotected communication channels, leaving the system vulnerable to several threats~\cite{kn:Johansson09}. As illustrative examples, we mention the cyber attacks on power transmission networks operated by SCADA systems reported in the public media~\cite{kn:WallStreet09}, and the Stuxnet malware that supposedly infected an industrial control system and disrupted its operation~\cite{kn:symantec2010_report,kn:Rid2011}.

There exists a vast literature on computer security focusing on three main properties of data and IT services, namely confidentiality, integrity, and availability~\cite{kn:Bishop2002}. Confidentiality relates to the non-disclosure of data by unauthorized parties. Integrity on the other hand concerns the trustworthiness of data, meaning there is no unauthorized change of the data contents or properties, while availability means that timely access to the data or system functionalities is ensured.
Unlike other IT systems where cyber-security mainly involves the protection of data, cyber attacks on networked control systems may influence physical processes through feedback actuation. Therefore networked control system security needs to consider threats at both the cyber and physical layers. Furthermore, it is of the utmost importance in the study of cyber attacks on control systems to capture the adversary's resources and knowledge. Cyber threats can be captured in the attack space illustrated in Figure~\ref{fig:attack_space}, which depicts several attack scenarios as points. For instance, the eavesdropping attack and the denial-of-service (DoS) attack are indicated in the figure.

\begin{figure}[h]
 \def\svgwidth{8.5cm}
  \centering
  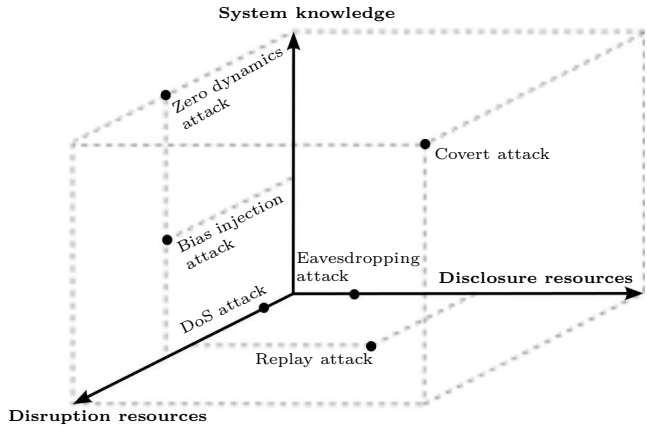
  \caption{The cyber-physical attack space.\label{fig:attack_space}}
\end{figure}

We propose three dimensions for the attack space: the adversary's \emph{a priori} system model knowledge and his disclosure and disruption resources. The \emph{a priori} system knowledge can be used by the adversary to construct more complex attacks, possibly harder to detect and with more severe consequences. Similarly, the disclosure resources enable the adversary to obtain sensitive information about the system during the attack by violating data confidentiality. Note that disclosure resources alone cannot disrupt the system operation. An example of an attack using only disclosure resources is the eavesdropping attack illustrated in Figure~\ref{fig:attack_space}. On the other hand, disruption resources can be used to affect the system operation, which happens for instance when data integrity or availability properties are violated. One such example is the DoS attack, where the data required for correctly operating the system are made unavailable. In particular this characterization fits the Stuxnet malware, which had resources to record and manipulate data in the SCADA network~\cite{kn:symantec2010_report}. Moreover, the complexity and operation of Stuxnet also indicate that its developers had access to a reasonable amount of knowledge of both physical and cyber components of the target control system.

\subsection{Related Work}

Control theory has contributed with frameworks to handle model uncertainties and disturbances as well as fault diagnosis and mitigation, see, for example, \cite{kn:Zhou1996} and \cite{Cheng_Patton_1999,Hwang2010}, respectively. These tools can be used to detect and attenuate the consequences of cyber attacks on networked control systems, as has recently been done in the literature.

Cyber attacks on control systems compromising measurement and actuator data integrity and availability have been considered in~\cite{kn:Cardenas08b}, where the authors modeled the attack effects on the physical dynamics. Several attack scenarios have been simulated and evaluated on the Tennessee-Eastman process control system~\cite{Cardenas2011} to study the attack impact and detectability. The attack scenarios in~\cite{Cardenas2011} are related to the ones considered in this paper, but we quantify the attack resources and policies in a systematic way.

Availability attacks have been analyzed in~\cite{AminCardenasSastry-HSCC-2009, kn:Gupta2010} for resource constrained adversaries with full-state information. Particularly, the authors considered DoS attacks in which the adversary could tamper with the communication channels and prevent measurement and actuator data from reaching their destination, rendering the data unavailable. A particular instance of the DoS attack in which the adversary does not have any \emph{a priori} system knowledge, as the attack in~\cite{AminCardenasSastry-HSCC-2009}, is represented in the attack space in Figure~\ref{fig:attack_space}.

Deception attacks compromising integrity have recently received attention. Replay attacks on the sensor measurements, which is a particular kind of deception attack, have been analyzed in~\cite{kn:Bruno09}. The authors considered the case where all the existing sensors were attacked and suitable counter-measures to detect the attack were proposed. In this attack scenario the adversary does not have any system knowledge but is able to access and corrupt the sensor data, thus having disclosure and disruptive resources, as depicted in Figure~\ref{fig:attack_space}.

Another class of deception attacks, false-data injection attacks, has been studied in recent work. For instance, in the case of power networks, an adversary with perfect model knowledge has been considered in~\cite{kn:Liu09}. The work in~\cite{kn:Kosut10} considered stealthy attacks with limited resources and proposed improved detection methods, while~\cite{kn:Sandberg10} analyzed the minimum number of sensors required for stealthy attacks. A corresponding measurement security metric for studying sets of vulnerable sensors was proposed in~\cite{kn:Sandberg10}. The consequences of these attacks have also been analyzed in~\cite{kn:Xie10,kn:Teixeira2011,kn:Teixeira_ACC2012}. In particular, in \cite{kn:Teixeira2011} the authors analyzed attack policies with limited model knowledge and performed experiments on a power system control software, showing that such attacks are stealthy and can induce the erroneous belief that the system is at an unsafe state. The models used in the previous work are static, hence these attack scenarios are closest to the bias injection attack shown in Figure~\ref{fig:attack_space}.

Data injection attacks on dynamic control systems were also considered.
In~\cite{Smith-IFAC-2011} the author characterizes the set of attack policies for covert (undetectable) false-data injection attacks with detailed model knowledge and full access to all sensor and actuator channels, while~\cite{kn:Pasqualetti2011} described the set of undetectable false-data injection attacks for omniscient adversaries with full-state information, but possibly compromising only a subset of the existing sensors and actuators. In the context of multi-agent systems, optimal adversary policies for data injection using full model knowledge and state information were derived in~\cite{KhanaferTouriBasar2012}. In these attack scenarios confidentiality was violated, as the adversary had access to either measurement and actuator data or full-state information. These attacks are therefore placed close to the covert attack in Figure~\ref{fig:attack_space}.

Most of the recent work on cyber-security of control systems has considered scenarios where the adversary has access to a large set of resources and knowledge, thus being placed far from the origin of the attack space in Figure~\ref{fig:attack_space}. A large part of the attack space has not been addressed. In particular, the class of detectable attacks that do not trigger conventional alarms has yet to be covered in depth.

\subsection{Contributions and Outline}
In this paper we consider a typical networked control architecture under both cyber and physical attacks. A generic adversary model applicable to several attack scenarios is discussed and the attack resources are mapped to the corresponding dimensions of the attack space. To illustrate the proposed framework, we consider several attack scenarios where the adversary's goal is to drive the system to an unsafe state while remaining stealthy. For each scenario we formulate the corresponding stealthy attack policy, comment on the attack's performance, and describe the adversary's capabilities along each dimension of the attack space in Figure~\ref{fig:attack_space}, namely the disclosure resources, disruption resources, and system knowledge. Some of the attack scenarios analyzed in the paper have been staged on a wireless quadruple tank testbed for security of control systems. The testbed architecture and results from the staged attacks are presented and discussed.

One of the attack scenarios analyzed corresponds to a novel type of detectable attack, the bias injection attack. Although this attack may be detected, it can drive the system to an unsafe region and it only requires limited model knowledge and no information about the system state. Stealthiness conditions for this attack are provided, as well as a methodology to assess the attack impact on the physical state of the system.

The material in this paper is an extension of the authors' preliminary work, see~\cite{kn:Teixeira_HICONS2012}. Particularly, in the current work the attack goals are formalized using the notion of safe regions of the state space and two additional attack scenarios are described and analyzed. Furthermore, the attack performance of each scenario is analyzed in more detail and additional results for the zero-dynamics and bias injection attacks are presented.

The outline of the paper is as follows. The system architecture and model are described in Section~\ref{sec:cps}, while Section~\ref{sec:attack_models} contains the adversary model and a detailed description of the attack resources on each dimension of the attack space. The framework introduced in the previous sections is then illustrated for five particular attack scenarios in Section~\ref{sec:attack_scenarios}, supposing that the adversary aims at driving the system to an unsafe state while remaining stealthy. The attack policy, attack performance, and required system knowledge, disclosure, and disruption resources are described in detail for each attack scenario. The results of the experiments for four of the attack scenarios in a secure control systems testbed are presented and discussed in Section~\ref{sec:experiments}, followed by conclusions in Section~\ref{sec:conc}.

\section{Networked Control System}\label{sec:cps}

In this section we describe the networked control system structure, where we consider three main components: the physical plant and communication network, the feedback controller, and the anomaly detector.

\subsection{Physical Plant and Communication Network}
The physical plant is modeled in a discrete-time state-space form
\begin{equation}\label{eq:plant_state space_faults}
\mathcal{P}:\left\{\begin{aligned}
x_{k+1}&=A x_k+B \tilde{u}_k + G w_k + F f_k\\
y_k&=C x_k + v_k %+ E f_k
\end{aligned}\right. ,
\end{equation}
where $x_k \in \mathbb{R}^{n}$ is the state variable, $\tilde{u}_k\in\mathbb{R}^{q}$ the control actions applied to the process, $y_k\in\mathbb{R}^{p}$ the measurements from the sensors at the sampling instant $k \in \mathbb{Z}$, and $f_k\in\mathbb{R}^d$ is the unknown signal representing the effects of anomalies, usually denoted as fault signal in the fault diagnosis literature~\cite{Ding2008}. The process and measurement noise, $w_k \in \mathbb{R}^n$ and $v_k \in \mathbb{R}^p$, represent the discrepancies between the model and the real process, due to unmodeled dynamics or disturbances, for instance, and we assume their means are respectively bounded by $\delta_w$ and $\delta_v$, i.e. $\bar{w} = \|\mathbb{E}\{w_k\}\| \leq \delta_w$ and $\bar{v} = \|\mathbb{E}\{v_k\}\|\leq \delta_v$.

The physical plant operation is supported by a communication network through which the sensor measurements and actuator data are transmitted, which at the plant side correspond to $y_k$ and $\tilde{u}_k$, respectively. At the controller side we denote the sensor and actuator data by $\tilde{y}_k\in\mathbb{R}^{p}$ and $u_k\in\mathbb{R}^{q}$, respectively. Since the communication network may be unreliable, the data exchanged between the plant and the controller may be altered, resulting in discrepancies in the data at the plant and controller ends. In this paper we do not consider the usual communication network effects such as packet losses and delays. Instead we focus on data corruption due to malicious cyber attacks, as described in Section~\ref{sec:attack_models}. Therefore the communication network \textit{per se} is supposed to be reliable, not affecting the data flowing through it.

Given the physical plant model~\eqref{eq:plant_state space_faults} and assuming an ideal communication network, the networked control system is said to have a \emph{nominal behavior} if $f_k = 0$, $\tilde{u}_k=u_k$, and $\tilde{y}_k=y_k$.
The absence of either one of these condition results in an abnormal behavior of the system.

\subsection{Feedback Controller}
In order to comply with performance requirements in the presence of the unknown process and measurement noises, we consider that the physical plant is controlled by an appropriate linear time-invariant feedback controller~\cite{kn:Zhou1996}. The output feedback controller can be written in a state-space form as
\begin{equation}\label{eq:controller_space_state}
\mathcal{F}:\left\{\begin{aligned}
z_{k+1} &= A_c z_k + B_c \tilde{y}_k \\
u_k &= C_c z_k + D_c \tilde{y}_k
\end{aligned}\right.
\end{equation}
where the states of the controller, $z_k \in \mathbb{R}^m$, may include the process state and tracking error estimates. Given the plant and communication network models, the controller is supposed to be designed so that acceptable performance is achieved under nominal behavior.

\subsection{Anomaly Detector}
In this section we consider the anomaly detector that monitors the system to detect possible anomalies, i.e. deviations from the nominal behavior. The anomaly detector is supposed to be collocated with the controller, therefore it only has access to $\tilde{y}_k$ and $u_k$ to evaluate the behavior of the plant.

Several approaches to detecting malfunctions in control systems are available in the fault diagnosis literature \cite{Ding2008,Hwang2010}. Here we consider the following observer-based Fault Detection Filter
\begin{equation}\label{eq:residual_dynamics}
\mathcal{D}:\left\{\begin{aligned}
\hat{x}_{k|k} & = A\hat{x}_{k-1|k-1} + Bu_{k-1} + K(\tilde{y}_k - \hat{y}_{k|k-1})\\
r_k  & = V(\tilde{y}_k - \hat{y}_{k|k})
\end{aligned} \right. ,
\end{equation}
where $\hat{x}_{k|k}\in\mathbb{R}^{n}$ and $\hat{y}_{k|k}=C\hat{x}_{k|k}\in\mathbb{R}^{p}$ are the state and output estimates given measurements up until time $k$, respectively, and $r_k\in\mathbb{R}^{p_d}$ the residue evaluated to detect and locate existing anomalies.

The anomaly detector is designed by choosing $K$ and $V$ such that
\begin{enumerate}
    \item under nominal behavior of the system (i.e., $f_k = 0$, $u_k=\tilde{u}_k$, $y_k=\tilde{y}_k$), the expected value of the residue converges asymptotically to a neighborhood of zero, i.e., $\lim_{k\rightarrow\infty} \|\mathbb{E}\{r_k\}\| \leq \delta_r$, with $\delta_r \in \mathbb{R}^+$;
    \item the residue is sensitive to the anomalies ($f_k\not\equiv0$).
 \end{enumerate}

An alarm is triggered  if the residue meets
\begin{equation}\label{eq:residue_threshold}
\| r_k \| \geq \delta_r + \delta_{\alpha},
\end{equation}
where $\delta_{\alpha}\in \mathbb{R}^+$ is chosen so that the false alarm rate does not exceed a given threshold $\alpha\in[0,\, 1]$.

\section{Adversary Models}\label{sec:attack_models}
\begin{figure}
\centering
 \def\svgwidth{8cm}
  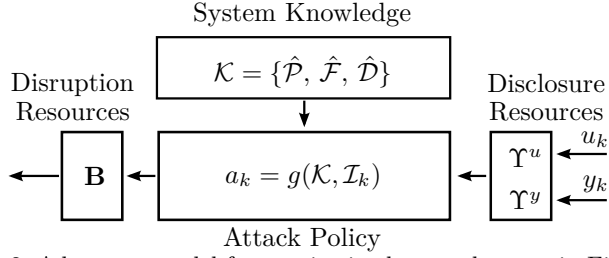
  \caption{Adversary model for a point in the attack space in Figure~\ref{fig:attack_space}.\label{fig:attack_policy}}
\end{figure}

The adversary model considered in this paper is illustrated in Figure~\ref{fig:attack_policy} and is composed of an attack policy and the adversary resources i.e., the system model knowledge, the disclosure resources, and the disruption resources. Each of the adversary resources can be mapped to a specific axis of the attack space in Figure~\ref{fig:attack_space}: $\mathcal{K}=\{\hat{\mathcal{P}}, \, \hat{\mathcal{F}}, \,\hat{\mathcal{D}}\}$ is the \emph{a~priori} system knowledge possessed by the adversary; $\mathcal{I}_k$ corresponds to the set of sensor and actuator data available to the adversary at time $k$ as illustrated in~\eqref{eq:disclosure_attack}, thus being mapped to the disclosure resources; $a_k$ is the attack vector at time $k$ that may affect the system behavior using the disruption resources captured by $\mathbf{B}$, as defined in the current section. The attack policy mapping $\mathcal{K}$ and $\mathcal{I}_k$ to $a_k$ at time $k$ is denoted as
\begin{equation}\label{eq:attack_policies}
\begin{aligned}
a_k &= g(\mathcal{K}, \mathcal{I}_k).
\end{aligned}
\end{equation}
Examples of attacks policies for different attack scenarios are given in Section~\ref{sec:attack_scenarios}.

In this section we describe the networked control system under attack with respect to the attack vector $a_k$. Then we detail the adversary's system knowledge, the disclosure resources, and the disruption resources. Models of the attack vector $a_k$ for particular disruption resources are also given.

\subsection{Networked Control System under Attack}
The system components under attack are now characterized for the attack vector $a_k$, which also includes the fault signal $f_k$. Considering the plant and controller states to be augmented as $\eta_k = [x_k^\top \quad z_k^\top ]^\top$, the dynamics of the closed-loop system composed by $\mathcal{P}$ and $\mathcal{F}$ under the effect of $a_k$ can be written as
\begin{equation}\label{eq:closed_loop_attacks}
\begin{aligned}
\eta_{k+1}
& =
\mathbf{A}
\eta_k +
\mathbf{B}
a_k
+
\mathbf{G}
\begin{bmatrix}
w_{k} \\
v_{k}
\end{bmatrix}\\
\tilde{y}_{k} & = \mathbf{C}\eta_{k} + \mathbf{D} a_{k}+ \mathbf{H} \begin{bmatrix}w_{k}\\  v_k\end{bmatrix},
\end{aligned}
\end{equation}
where the system matrices are
\begin{equation*}
\begin{array}{ll}
\mathbf{A} = \begin{bmatrix}
A + B D_c C  & B C_c\\
B_c C & A_c
\end{bmatrix}, &
\mathbf{G} = \begin{bmatrix}
G & B D_c  \\
0 & B_c
\end{bmatrix},\\
&\\
\mathbf{C} = \begin{bmatrix}
C  & 0
\end{bmatrix},
& \mathbf{H} = \begin{bmatrix}
0 & I  \\
\end{bmatrix},
\end{array}
\end{equation*}
and $\mathbf{B}$ and $\mathbf{D}$ capture the way in which the attack vector $a_k$ affects the plant and controller. These matrices are characterized for some attack scenarios in Section~\ref{sec:model_disruptive}.

Similarly, using $\mathcal{P}$ and $\mathcal{D}$ as in~\eqref{eq:plant_state space_faults} and~\eqref{eq:residual_dynamics}, respectively, the anomaly detector error dynamics under attack are described by
\begin{equation}\label{eq:residual_dynamics_attack}
\begin{aligned}
\xi_{k|k} & = \mathbf{A}_e\xi_{k-1|k-1} +\mathbf{B}_e a_{k-1} + \mathbf{G}_e \begin{bmatrix}w_{k-1}\\  v_k\end{bmatrix} \\
r_k  & = \mathbf{C}_e \xi_{k-1|k-1}   + \mathbf{D}_e a_{k-1}+ \mathbf{H}_e \begin{bmatrix}w_{k-1}\\  v_k\end{bmatrix},
\end{aligned}
\end{equation}
where $\xi_{k|k}\in\mathbb{R}^{n}$ is the estimation error and
\begin{equation*}
\begin{array}{ll}
\mathbf{A}_e = (I-KC)A,&\;
%&\mathbf{B}_e = \begin{bmatrix}(I-KC)F & (I-KC)B\Gamma^u & -K\Gamma^y & 0 \end{bmatrix},\\
\mathbf{G}_e = \begin{bmatrix}(I-KC)G & -K \end{bmatrix},\\
\mathbf{C}_e = VC(I-KC)A,&\;
%&\mathbf{D}_e = \begin{bmatrix}V(I-KC)F & V(I-KC)B\Gamma^u & V(I-CK)\Gamma^y & 0 \end{bmatrix},\\
\mathbf{H}_e = \begin{bmatrix}VC(I-KC)G & V(I-CK) \end{bmatrix}.
\end{array}
\end{equation*}
The matrices $\mathbf{B}_e$ and $\mathbf{D}_e$ are specific to the available disruptive resources and are characterized in Section~\ref{sec:model_disruptive}.

\subsection{System Knowledge}
The amount of \emph{a priori} knowledge regarding the control system is a core component of the adversary model, as it may be used, for instance, to render the attack undetectable. In general, we may consider that the adversary approximately knows the model of the plant ($\hat{\mathcal{P}}$) and the algorithms used in the feedback controller ($\hat{\mathcal{F}}$) and the anomaly detector ($\hat{\mathcal{D}}$), thus denoting the adversary knowledge by $\mathcal{K}=\{\hat{\mathcal{P}}, \hat{\mathcal{F}},\hat{\mathcal{D}}\}$. Figure~\ref{fig:attack_space} illustrates several types of attack scenarios with different amounts of required system knowledge. In particular, note that the replay attacks do not need any knowledge of the system components, thus having $\mathcal{K}=\emptyset$, while the covert attack requires full knowledge about the system, hence $\mathcal{K}=\{\mathcal{P}, \mathcal{F},\mathcal{D}\}$.

\subsection{Disclosure Resources}
The disclosure resources enable the adversary to gather sequences of data from the calculated control actions $u_k$ and the real measurements $y_k$ through disclosure attacks. Denote $\mathcal{R}^u\subseteq \{1,\dots,q\}$ and $\mathcal{R}^y\subseteq \{1,\dots,p\}$ as the disclosure resources, i.e. set of actuator and sensor channels that can be accessed during disclosure attacks, and let $\mathcal{I}_k$ be the control and measurement data sequence gathered by the adversary from time $k_0$ to $k$. The disclosure attacks can then be modeled as
\begin{equation}\label{eq:disclosure_attack}
\begin{aligned}
\mathcal{I}_k & := \mathcal{I}_{k-1}
\cup \left\{
\begin{bmatrix}
\Upsilon^u  &   0\\
0           &   \Upsilon^y
\end{bmatrix}
\begin{bmatrix}
u_k\\
y_k
\end{bmatrix} \right\}
,
\end{aligned}
\end{equation}
where  $\Upsilon^{u}\in\mathbb{B}^{ |\mathcal{R}^{u}|\times q}$ and $\Upsilon^{y}\in\mathbb{B}^{ |\mathcal{R}^{y}|\times p}$ are the binary incidence matrices mapping the data channels to the corresponding data gathered by the adversary and $\mathcal{I}_{k_0} = \emptyset$.

As seen in the above description of disclosure attacks, the physical dynamics of the system are not affected by these type of attacks. Instead, these attacks gather intelligence that may enable more complex attacks, such as the replay attacks depicted in Figure~\ref{fig:attack_space}.

\subsection{Disruption Resources}\label{sec:model_disruptive}
 As seen in the system dynamics under attack,~\eqref{eq:closed_loop_attacks} and~\eqref{eq:residual_dynamics_attack}, disruption resources are related to the attack vector $a_k$ and may be used to affect the several components of the system. The way a particular attack disturbs the system operation depends not only on the respective resources, but also on the nature of the attack. For instance, a physical attack directly perturbs the system dynamics, whereas a cyber attack disturbs the system through the cyber-physical couplings.
 To better illustrate this discussion we now consider physical and data deception attacks.

\subsubsection{Physical Resources}
Physical attacks may occur in control systems, often in conjunction with cyber attacks. For instance, in~\cite{kn:Amin2010_water} water was pumped out of an irrigation system while the water level measurements were corrupted so that the attack remained stealthy. Since physical attacks are similar to the fault signals $f_k$ in~\eqref{eq:plant_state space_faults}, in the following sections we consider $f_k$ to be the physical attack modifying the plant dynamics as
\begin{align*}
x_{k+1}&=Ax_k+B \tilde{u}_k + G w_k + Ff_k\\
y_k &= Cx_k.
\end{align*}

Considering $a_k = f_k$, the resulting system dynamics are described by~\eqref{eq:closed_loop_attacks} and~\eqref{eq:residual_dynamics_attack} with
%\begin{equation*}  % use when in double column
%\begin{array}{ll}
%\mathbf{B}=
%\begin{bmatrix}
%F\\
%0
%\end{bmatrix}, &
%\mathbf{D} = 0, \\
%& \\
%\mathbf{B}_e = (I-KC)F, &
%\mathbf{D}_e = VC(I-KC)F.
%\end{array}
%\end{equation*}
\begin{equation*}
\begin{array}{llll}
\mathbf{B}=
\begin{bmatrix}
F\\
0
\end{bmatrix},\; &
\mathbf{D} = 0,\; &
\mathbf{B}_e = (I-KC)F,\; &
\mathbf{D}_e = VC(I-KC)F.
\end{array}
\end{equation*}
Note that the disruption resources in this attack are captured in the matrix $F$.

\subsubsection{Data Deception Resources}
The deception attacks modify the control actions $u_k$ and sensor measurements $y_k$ from their calculated or real values to the corrupted signals $\tilde{u}_k$ and $\tilde{y}_k$, respectively. Denoting $\mathcal{R}_I^u\subseteq \{1,\dots,q\}$ and $\mathcal{R}_I^y\subseteq \{1,\dots,p\}$ as the deception resources, i.e. set of actuator and sensor channels that can be affected, the deception attacks are modeled as
%\begin{equation}\label{eq:deception_attack}
%\begin{aligned}
% \tilde{u}_k & :=  u_k + \Gamma^u b^u_k \\
% \tilde{y}_k & :=  y_k + \Gamma^y b^y_k
%\end{aligned}
%\end{equation}
\begin{equation}\label{eq:deception_attack}
 \tilde{u}_k  :=  u_k + \Gamma^u b^u_k,\quad
 \tilde{y}_k  :=  y_k + \Gamma^y b^y_k,
\end{equation}
where the signals $b^u_k \in \mathbb{R}^{|\mathcal{R}_I^{u}|}$ and $b^y_k \in \mathbb{R}^{|\mathcal{R}_I^{y}|}$ represent the data corruption and $\Gamma^{u}\in\mathbb{B}^{ q \times |\mathcal{R}_I^{u}|}$ and $\Gamma^{y}\in\mathbb{B}^{ p \times |\mathcal{R}_I^{y}|}$ ($\mathbb{B} := \left \{0, \, 1 \right \}$) are the binary incidence matrices mapping the data corruption to the respective data channels. The matrices $\Gamma^{u}$ and $\Gamma^{y}$ indicate which data channels can be accessed by the adversary and are therefore directly related to the adversary resources in deception attacks.

Defining  $a_k = [b^{u\top}_{k} \quad b^{y\top}_{k+1}\quad b^{y\top}_{k} ]^\top$, the system dynamics are given by~\eqref{eq:closed_loop_attacks} and~\eqref{eq:residual_dynamics_attack} with
%\begin{equation*}
%\begin{array}{l}
%\mathbf{B}=
%\begin{bmatrix}
%B\Gamma^u   & 0 &   B D_c\Gamma^y   \\
%0           & 0 &   B_c\Gamma^y
%\end{bmatrix}, \quad
%\mathbf{D} = \begin{bmatrix}
%0   & 0 &   \Gamma^y   \\
%\end{bmatrix}, \\
%\\
%\mathbf{B}_e = \begin{bmatrix} (I-KC)B\Gamma^u & -K\Gamma^y & 0 \end{bmatrix},\\
%\mathbf{D}_e = \begin{bmatrix} VC(I-KC)B\Gamma^u & V(I-CK)\Gamma^y & 0 \end{bmatrix}.
%\end{array}
%\end{equation*}
\begin{equation*}
\begin{array}{l}
\mathbf{B}=
\begin{bmatrix}
B\Gamma^u   & 0 &   B D_c\Gamma^y   \\
0           & 0 &   B_c\Gamma^y
\end{bmatrix}, \quad
\mathbf{D} = \begin{bmatrix}
0   & 0 &   \Gamma^y   \\
\end{bmatrix}, \quad
\mathbf{B}_e = \begin{bmatrix} (I-KC)B\Gamma^u & -K\Gamma^y & 0 \end{bmatrix},\quad
\mathbf{D}_e = \begin{bmatrix} VC(I-KC)B\Gamma^u & V(I-CK)\Gamma^y & 0 \end{bmatrix}.
\end{array}
\end{equation*}

Note that deception attacks do not possess any disclosure capabilities, as depicted in Figure~\ref{fig:attack_space} for examples of deception attacks such as the bias injection attack.

\section{Attack Scenarios}\label{sec:attack_scenarios}

In this section we discuss the general goal of an adversary and likely choices of the attack policy $g(\cdot,\cdot)$. In particular, using the framework introduced in the previous sections, we consider several attack scenarios where the adversary's goal is to drive the system to an unsafe state while remaining stealthy. For each scenario we formulate the corresponding stealthy attack policy, comment on the attack's performance, and describe the adversary's capabilities along each dimension of the attack space in Figure~\ref{fig:attack_space}, namely the disclosure resources, disruption resources, and system knowledge. A set of these scenarios is illustrated by experiments on a process control testbed in Section~\ref{sec:experiments}.

\subsection{Attack Goals and Constraints}
In addition to the attack resources, the attack scenarios need to also include the intent of the adversary, namely the attack goals and constraints shaping the attack policy. The attack goals can be stated in terms of the attack impact on the system operation, while the constraints may be related to the attack detectability.

%\textbf{Attack Goals: }
Several physical systems have tight operating constraints which if not satisfied might result in physical damage to the system. In this work we use the concept of safe regions to characterize the safety constraints.

\begin{definition}\label{def:safe_set}
 At a given time instant $k$, the system is said to be safe if $x_{k}\in \mathcal{S}_x$, where $\mathcal{S}_x$ is a closed and compact set with non-empty interior.
\end{definition}

\begin{assumption}
The system is in a safe state at the beginning of the attack, i.e. $x_{k_0}\in \mathcal{S}_x$.
\end{assumption}

The physical impact of an attack can be evaluated by assessing whether or not the state of the system remained in the safe set during and after the attack. The attack is considered successful if the state is driven out of the safe set.

%\textbf{Attack Constraints: }
Regarding the attack constraints, we consider that attacks are constrained to remain stealthy. Furthermore, we consider the disruptive attack component consists of only physical and data deception attacks, and thus we have the attack vector $a_k = [f_k^\top \quad b^{u\top}_{k} \quad b^{y\top}_{k+1}\quad b^{y\top}_{k} ]^\top$. Given the anomaly detector described in Section~\ref{sec:cps} and denoting $\mathcal{A}_{k_0}^{k_f}=\{a_{k_0},\, \dots,\, a_{k_f}\}$ as the attack signal, the set of stealthy attacks are defined as follows.
\begin{definition}\label{def:stealthy}
The attack signal $\mathcal{A}_{k_0}^{k_f}$ is stealthy if $\|r_k\| < \delta_r + \delta_\alpha,\; \forall k\geq k_0$.
\end{definition}
Note that the above definition is dependent on the initial state of the system at $k_0$, as well as the noise terms $w_k$ and $v_k$.

Since the closed-loop system~\eqref{eq:closed_loop_attacks} and the anomaly detector~\eqref{eq:residual_dynamics_attack} under linear attack policies are linear systems, each of these systems can be separated into two components, the nominal component with $a_k=0 \; \forall k$ and the following systems
\begin{equation}\label{eq:closed_loop_attacks_linear}
\begin{aligned}
\eta_{k+1}^a
& =
\mathbf{A} \eta_k^a +
\mathbf{B} a_k\\
\tilde{y}^a_k &= \mathbf{C}\eta_{k}^a + \mathbf{D} a_k
\end{aligned}
\end{equation}
and
\begin{equation}\label{eq:residual_dynamics_attack_linear}
\begin{aligned}
\xi_{k|k}^a & = \mathbf{A}_e\xi_{k-1|k-1}^a +\mathbf{B}_e a_{k-1} \\
r^a_k  & = \mathbf{C}_e \xi_{k-1|k-1}^a   + \mathbf{D}_e a_{k-1},
\end{aligned}
\end{equation}
with $\eta_0^a = \xi_{0|0}^a=0$.

Assuming the system is behaving nominally before the attack, using the triangle inequality and linearity of~\eqref{eq:residual_dynamics_attack} we have $||r^a_k||\leq \delta_\alpha \Rightarrow ||r_k||\leq\delta_r + \delta_\alpha$, leading to the following definition:
\begin{definition}\label{def:alpha_stealthy}
The attack signal $\mathcal{A}_{k_0}^{k_f}$ is $\alpha-$stealthy with respect to $\mathcal{D}$ if $||r^a_k|| < \delta_\alpha, \; \forall k\geq k_0$.
\end{definition}
Albeit more conservative than Definition~\ref{def:stealthy}, this definition only depends on the attack signals $\mathcal{A}_{k_0}^{k_f}$. Similarly, the impact of attacks on the closed-loop system can also be analyzed by looking at the linear system~\eqref{eq:closed_loop_attacks_linear}, as illustrated in Section~\ref{sec:attack_bias} for the bias injection attack.

\subsection{Denial-of-Service Attack}
The DoS attacks prevent the actuator and sensor data from reaching their respective destinations and should therefore be modeled as the absence of data, for instance $u_k = \emptyset$ if all the actuator data is unavailable. However such a model would not fit the framework in~\eqref{eq:closed_loop_attacks} and~\eqref{eq:residual_dynamics_attack} where $a_k$ is assumed to be a real valued vector. Hence we consider instead one of the typical mechanisms used by digital controllers to deal with the absence of data~\cite{kn:SchenatoTAC2009}, in which the absent data is replaced with the last received data, $u_{\tau_u}$ and $y_{\tau_y}$ respectively. Denoting $\mathcal{R}_A^u\subseteq \{1,\dots,q\}$ and $\mathcal{R}_A^y\subseteq \{1,\dots,p\}$ as the set of actuator and sensor channels that can be made unavailable, we can model DoS attacks as deception attacks in~\eqref{eq:deception_attack} with
\begin{equation}\label{eq:DoS_attack}
\begin{matrix}
b^u_k & :=  -S^u_k \Gamma^{u\top} (u_k - u_{\tau_u})\\
b^y_k & :=  -S^y_k \Gamma^{y\top} (y_k - y_{\tau_y})
\end{matrix}
\end{equation}
where $S^u_k\in\mathbb{B}^{|\mathcal{R}_A^{u}| \times |\mathcal{R}_A^{u}|}$ and $S^y_k\in\mathbb{B}^{|\mathcal{R}_A^{y}|\times |\mathcal{R}_A^{y}|}$ are boolean diagonal matrices where the $i-$th diagonal entry indicates whether a DoS attack is performed ($[S^{(\cdot)}_k]_{ii}=1$) or not ($[S^{(\cdot)}_k]_{ii}=0$) on the corresponding channel. Therefore DoS attacks on the data are a type of disruptive attacks, as depicted in Figure~\ref{fig:attack_space}.

\textbf{Attack policy: }
The attack scenario analyzed in this paper considers a Bernoulli adversary~\cite{AminCardenasSastry-HSCC-2009} on the sensor channels following the random policy
\begin{equation*}
\begin{aligned}
\mathbb{P}([S^{y}_k]_{ii} = 1) &= 0,\; \forall i=1,\dots, |\mathcal{R}_A^{u}|,\quad  k < k_0  \\
\mathbb{P}([S^{y}_k]_{ii} = 1) &= p,\; \forall i=1,\dots, |\mathcal{R}_A^{u}|,\quad  k\geq k_0
\end{aligned}
\end{equation*}
where $p$ is the probability of blocking the data packet at any given time.

\textbf{Attack performance: }
Although the absence of data packets is not stealthy since it is trivially detectable, DoS attacks may be misdiagnosed as a poor network condition. As for the impact on the closed-loop system, the results available for Bernoulli packet losses readily apply to the current attack scenario~\cite{kn:Zhang2001_NCS,kn:Schenato07foundationsNCS,kn:SchenatoTAC2009}. In particular, we recall a result for the case where a hold scheme~\eqref{eq:DoS_attack} is used in the absence of data.

\begin{proposition}[Theorem 8 in~\cite{kn:Zhang2001_NCS}]\label{thm:DoS_stability}
Assume the closed-loop system with no DoS attack is stable. Then the closed-loop system with Bernoulli DoS attacks is exponentially stable for $p\in[0,\; 1)$ if the open-loop system
\[
\eta_{k+1}=\begin{bmatrix}
A & BC_c\\
0 & A_c
\end{bmatrix}\eta_{k}
\]
is marginally stable.
\end{proposition}

\textbf{Disclosure resources: }
Although the proposed model of DoS attacks in \eqref{eq:DoS_attack} contains the control and output signals, note that no disclosure resources are needed in the actual implementation of the attack. Thus we have $\mathcal{R}^{u}=\mathcal{R}^{y}=\emptyset$.

\textbf{Disruption resources: }
The disruption capabilities correspond to the data channels that the adversary is able to make unavailable, $\mathcal{R}_A^{u}$ and $\mathcal{R}_A^{y}$.

\textbf{System knowledge: }
For the Bernoulli attack policy, no \emph{a priori} knowledge of the system model is needed.

\subsection{Replay Attack}\label{sec:replay_attack}
In replay attacks the adversary first performs a disclosure attack from $k=k_0$ until $k_r$, gathering sequences of data $\mathcal{I}_{k_r}$, and then begins replaying the recorded data at time $k=k_r+1$ until the end of the attack at $k=k_f>k_r$, as illustrated in Figure~\ref{fig:replay_attack}. In the scenario considered here the adversary is also able to perform a physical attack while replaying the recorded data, which covers the experiment on a water management SCADA system reported in~\cite{kn:Amin2010_water} and one of Stuxnet's operation mode~\cite{kn:symantec2010_report}.
\begin{figure}[ht]
  \centering
\subfigure[Phase I of the replay attack~\eqref{eq:replay_policy0}.]{
\def\svgwidth{6.7cm}
    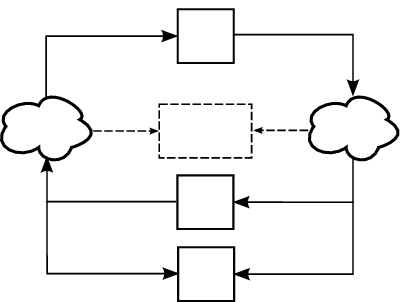
\label{fig:replay_attack1}
}
\vspace{2pt}
\subfigure[Phase II of the replay attack~\eqref{eq:replay_policy}.]{
\def\svgwidth{6.7cm}
    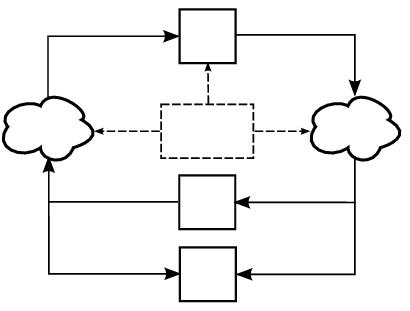
\label{fig:replay_attack2}
}
\caption{Schematic of the replay attack.\label{fig:replay_attack}}
\end{figure}

\textbf{Attack policy: }
Similar to the work in~\cite{kn:Bruno09}, assuming $\mathcal{R}^{(\cdot)} = \mathcal{R}_I^{(\cdot)}$ i.e., the adversary can corrupt the digital channels from which the data sequences are gathered, the replay attack policy can be described as
\begin{equation}\label{eq:replay_policy0}
\mbox{Phase I: }\; \left\{\begin{aligned}
a_k &= 0\\
\mathcal{I}_{k} &= \mathcal{I}_{k-1} \cup
\left\{
\begin{bmatrix}
\Upsilon^u  &   0 \\
0           &   \Upsilon^y
\end{bmatrix}
\begin{bmatrix}
u_{k} \\
y_{k}
\end{bmatrix} \right\},
\end{aligned}\right.
\end{equation}
with $k_0 \leq k\leq k_r$ and $\mathcal{I}_{k_0}=\emptyset$ and
\begin{equation}\label{eq:replay_policy}
\mbox{Phase II: }\; \left\{\begin{aligned}
a_k  &=
\begin{bmatrix}
g_f(\mathcal{K}, \mathcal{I}_{k_r}) \\
\Upsilon^u (u_{k - T} - u_k)\\
\Upsilon^y (y_{k+1 - T}-y_{k+1}) \\
\Upsilon^y (y_{k - T}-y_k)
\end{bmatrix}\\
\mathcal{I}_{k} &= \mathcal{I}_{k-1},
\end{aligned}\right.
\end{equation}
where $T=k_r-1 + k_0$ and $ k_r+1 \leq k\leq k_f$.
An interesting instance of this attack scenario consists of applying a pre-defined physical attack to the plant, while using replay attacks to render the attack stealthy. In this case the physical attack signal $f_k$ corresponds to an open-loop signal, $f_k = g_f(k)$.

\textbf{Attack performance: }
The work in~\cite{kn:Bruno09} provided conditions under which replay attacks with access to all measurement data channels are stealthy. However, these attacks are not guaranteed to be stealthy when only a subset of the data channels is attacked. In this case, the stealthiness constraint may require additional knowledge of the system model. For instance, the experiment presented in Section~\ref{sec:experiments} requires knowledge of the physical system structure, so that $f_k$ only excites the attacked measurements. Hence $f_k$ can be seen as a zero-dynamics attack with respect to the healthy measurements, which is characterized in the section below. Since the impact of the replay attack is dependent only on $f_k$, we refer the reader to Section~\ref{sec:attack_zero} for a characterization of the replay attack's impact.

\textbf{Disclosure resources: }
The disclosure capabilities required to stage this attack correspond to the data channels that can be eavesdropped by the attacks, namely $\mathcal{R}^{u}$ and $\mathcal{R}^{y}$.

\textbf{Disruption resources: }
In this case the deception capabilities correspond to the data channels that the adversary can tamper with, $\mathcal{R}_I^{u}$ and $\mathcal{R}_I^{y}$. In particular, for replay attacks the adversary can only tamper with the data channels from which data has been previously recorded, i.e. $\mathcal{R}_I^{u} \subseteq \mathcal{R}^{u}$ and $\mathcal{R}_I^{y}\subseteq \mathcal{R}^{y}$.

Direct disruption of the physical system through the signal $f_k$ depends on having direct access to the physical system, modeled by the matrix $F$ in~\eqref{eq:plant_state space_faults}.

\textbf{System knowledge: }
Note that no \emph{a priori} knowledge $\mathcal{K}$ on the system model is needed for the cyber component of the attack, namely the data disclosure and deception attack, as seen in the attack policy~\eqref{eq:replay_policy0} and~\eqref{eq:replay_policy}. As for the physical attack, $f_k$, the required knowledge is scenario dependent. In the scenario considered in the experiments described in Section~\ref{sec:experiments}, this component was modeled as an open-loop signal, $f_k = g_f(k)$.

\subsection{Zero-Dynamics Attack}\label{sec:attack_zero}
Recalling that for linear attack policies the plant and the anomaly detector are linear systems, \eqref{eq:closed_loop_attacks_linear} and~\eqref{eq:residual_dynamics_attack_linear} respectively, Definition~\ref{def:alpha_stealthy} states that this type of attacks are $0-$stealthy if $r^a_k=0,\, k=k_0,\dots,k_f$. The idea of $0-$stealthy attacks then consists of designing an attack policy and attack signal $\mathcal{A}_{k_0}^{k_f}$ so that the residue $r_k$ does not change due to the attack.

A particular subset of $0-$stealthy attacks are characterized in the following lemma:
\begin{lemma}\label{lem:output_zeroing_attack}
The attack signal $\mathcal{A}_{k_0}^{k_f}$ is $0-$stealthy with respect to any $\mathcal{D}$ if $\tilde{y}^a_k = 0, \, \forall k\geq k_0$.
\end{lemma}
\begin{proof}
Consider the attacked components of the controller and the anomaly detector in~\eqref{eq:closed_loop_attacks_linear} and~\eqref{eq:residual_dynamics_attack_linear} with $\hat{x}^a_0=\xi_{0|0}^a=0$. From the controller dynamics it directly follows that $\tilde{y}^a_k = 0, \, \forall k\geq k_0$ results in $u_k^a = 0, \, \forall k\geq k_0$, as the input to the controller ($\tilde{y}^a_k$) is zero. Since $\hat{x}^a_0=0$ and $\tilde{y}^a_k = u_k^a = 0, \, \forall k\geq k_0$, meaning that the detector's inputs are zero, we then conclude $r^a_k = 0, \, \forall k\geq k_0$.
\end{proof}

Both the definition of $0-$stealthy attacks and Lemma~\ref{lem:output_zeroing_attack} indicate that these attacks are decoupled from the outputs of linear systems, $r_k$ and $y_k$ respectively. Hence finding $0-$stealthy attack signals relates to the output zeroing problem or zero-dynamics studied in the control theory literature~\cite{kn:Zhou1996}. Note that such an attack requires the perfect knowledge of the plant dynamics $P$ and the attack signal is then based on the open-loop prediction of the output changes due to the attack. This is illustrated in Figure~\ref{fig:zero_attack} where $\mathcal{K}_z$ denote the zero-dynamics and there is no disclosure of sensor or actuator data.

\begin{figure}[ht]
 \def\svgwidth{6.7cm}
  \centering
  \input{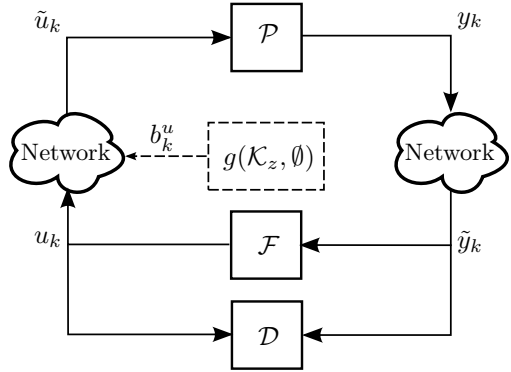}
  \caption{Schematic of the zero-dynamics attack.\label{fig:zero_attack}}
\end{figure}

\textbf{Attack policy: }
The attack policy then corresponds to the input sequence ($a_k$) that makes the outputs of the process ($\tilde{y}^a_k$) identically zero for all $k$ and is illustrated in Figure~\ref{fig:zero_attack}. It can be shown~\cite{kn:Zhou1996} that the solution to this problem is given by the sequence
\begin{equation} \label{zero_sequence}
a_k = g \nu^k,
\end{equation}
parameterized by the input-zero direction $g$ and the system zero $\nu$.

For sake of simplicity we consider a particular instance of this attack, where only the actuator data is corrupted.
In this case the zero attack policy corresponds to the transmission zero-dynamics of the plant. The plant dynamics due to an attack on the actuator data are described by
\begin{equation}\label{eq:transmission_system}
\begin{aligned}
x_{k+1}^a &= A x_k^a + B a_k\\
\tilde{y}^a_k &= C x_k^a
\end{aligned}
\end{equation}
with $a_k = b^u_k$. Given the discrete-time system~\eqref{eq:transmission_system} with $B$ having full column rank, the transmission zeros can be calculated as the values $\nu \in \mathbb{C}$ that cause the matrix $P(\nu)$ to lose rank, where
\begin{equation*}%\label{eq:transmission_zeros}
P(\nu)=\begin{bmatrix} \nu I - A & -B \\ C & 0 \end{bmatrix}.
\end{equation*}
Those values are called minimum phase or non-minimum phase zeros depending on whether they are stable or unstable zeros, respectively. In discrete-time systems a zero is stable if $|\nu| < 1$ and unstable otherwise.

The input zero direction can be obtained by solving the following equation
\begin{equation}\label{eq:zero_dynamics}
\begin{bmatrix} \nu I-A & -B \\ C & 0 \end{bmatrix} \begin{bmatrix} x_0 \\ g \end{bmatrix} = \begin{bmatrix} 0 \\ 0 \end{bmatrix},
\end{equation}
where $x_0$ is the initial state of the system for which the input sequence~\eqref{zero_sequence} results in an identically zero output, $\tilde{y}^a_k=0\,\forall k$.

\begin{lemma}\label{lem:zero_dynamics_invariance}
 Let $x_0$ be the initial state of the system, where $x_0$ satisfies~\eqref{eq:zero_dynamics}. The state trajectories generated by the zero-dynamics attack are contained in $\mbox{span}(x_0)$ i.e., $x_{k}^a\in \mbox{span}(x_0)\; \forall k\geq0$.
\end{lemma}
\begin{proof}
Consider the zero-dynamics attack parameterized by $x_0$ and $g$ and denote $L$ as a map for which $L x_0 = g$. Then from~\eqref{eq:zero_dynamics} we have $\left(\nu I -(A+BL)\right) x_0 = 0$ and conclude that $x_0$ is an eigenvector of $A+BL$ associated with its eigenvalue $\nu$. Now consider the state evolution under attack, $x^a_{k+1} = Ax^a_k + Bg$ with $x^a_0=x_0$. The proof is completed by noting that $x^a_1 = Ax_0 + Bg = (A+BL)x_0 = \nu x_0$ and applying an induction argument.
\end{proof}

\textbf{Attack performance: }
Note that the zero-dynamics attack is $0-$stealthy only if $x^a_0 = x_0$. However the initial state of the system under attack $x^a_0$ is defined to be zero at the beginning of the attack. Therefore stealthiness of the attack may be violated for large differences between $x^a_0=0$ and $x_0$. We refer the reader to~\cite{Teixeira_Allerton2012} for a detailed analysis of the effects of zero initial conditions on zero-dynamics attacks.

If the zero is stable, that is $| \nu  | < 1$, the attack will asymptotically decay to zero, thus having little effect on the plant. However, in the case of unstable zeros the attack grows geometrically, which could cause a great damage to the process. This statement is captured in the following result.

\begin{theorem}\label{thm:zero_unsafe}
A zero-dynamics attack with $| \nu  | > 1$ leads the system to an unsafe state if and only if $\mbox{span}(x_0)$ is not contained in $\mathcal{S}_x$.
\end{theorem}
\begin{proof}
Follows directly from Lemma~\ref{lem:zero_dynamics_invariance} and from the fact that the zero-attack with $|\nu| >1$ generates an unstable state trajectory moving away from the origin along $\mbox{span}(x_0)$.
\end{proof}

\textbf{Disclosure resources: }
This attack scenario considers an open-loop attack policy and so no disclosure capabilities are required, resulting in $\mathcal{R}^{u}=\mathcal{R}^{y}=\emptyset$ and $\mathcal{I}^u_{k}=\mathcal{I}^y_{k}=\emptyset \,\forall k$.

\textbf{Disruption resources: }
The disruption capabilities in this attack scenario correspond to the ability of performing deception attacks on the actuator data channels. Therefore the required resources are $\mathcal{R}_I^{u}=\{1,\dots,q\}$, $\mathcal{R}_I^{y}=\emptyset$, and $F=0$

\textbf{System knowledge: }
The ability to compute the open-loop attack policy requires the perfect knowledge zero-dynamics, which we denote as $\mathcal{K}_z$. Note that computing the zero-dynamics requires perfect knowledge of the plant dynamics, namely $A$, $B$, and $C$. No knowledge of the feedback controller or anomaly detector is assumed in this scenario.

\subsection{Local Zero-Dynamics Attack}
In the previous scenario the zero-dynamics attack was characterized in terms of the entire system. Here we further restrict the adversary resources by considering that the adversary has disruption resources and knows the model of only a subset of the system. In particular, we rewrite the plant dynamics~\eqref{eq:transmission_system} as
\begin{equation}\label{eq:transmission_system_local}
\begin{aligned}
\begin{bmatrix}
x_{k+1}^{1} \\
x_{k+1}^{2}
\end{bmatrix}
&=
\begin{bmatrix}
A_{11} & A_{12} \\
A_{21} & A_{22}
\end{bmatrix}
\begin{bmatrix}
x_{k}^{1} \\
x_{k}^{2}
\end{bmatrix}
+
\begin{bmatrix}
B_{1}\\
0
\end{bmatrix}
a_k\\
\tilde{y}^a_k &= \begin{bmatrix}
C_{1} & C_{2}
\end{bmatrix}
\begin{bmatrix}
x_{k}^{1} \\
x_{k}^{2}
\end{bmatrix}
\end{aligned}
\end{equation}
and assume the adversary has access to only $A_{11}$, $A_{21}$, $B_1$, and $C_1$. From the adversary's view, this local system is characterized by
\begin{equation*}
\begin{aligned}
x_{k+1}^1 &= A_{11} x_k^1 + B_1 a_k + A_{12}x_k^2\\
y^l_k &= \begin{bmatrix}
C_1 \\
A_{21}
\end{bmatrix}x_k^1,
\end{aligned}
\end{equation*}
where $y^l_k$ encodes the measurements depending on the local state, $C_1x^1_k$, and the interaction between the local subsystem and the remaining subsystems, $A_{21}x^1_k$.

\textbf{Attack policy: }
Similar to the zero-dynamics attack, the attack policy is given by the sequence
\begin{equation*}
a_k = g \nu^k,
\end{equation*}
where $g$ is the input zero direction for the chosen zero $\nu$. The input zero direction can be obtained by solving the following equation
\begin{equation*}%\label{eq:zero_dynamics_local}
\begin{bmatrix} \nu I-A_{11} & -B_1 \\ C_1 & 0 \\ A_{21} & 0\end{bmatrix} \begin{bmatrix} x^1_0 \\ g^1 \end{bmatrix} = \begin{bmatrix} 0 \\ 0\\ 0 \end{bmatrix}.
\end{equation*}

Note that the zero-dynamics parameterized by $g^1$ and $\nu$ correspond to local zero-dynamics of the global system.

\textbf{Attack performance: }
A similar discussion as for the global zero-dynamics attack applies to this scenario. In particular, the stealthiness of the local zero-dynamics attack may be violated for large differences between $x^1_0$ and $0$. Additionally, as stated in Theorem~\ref{thm:zero_unsafe}, attacks associated with unstable zeros yielding $|\nu|>1$ are more dangerous and may lead the system to an unsafe state.

\textbf{Disclosure resources: }
This attack scenario considers an open-loop attack policy and so no disclosure capabilities are required, resulting in $\mathcal{R}^{u}=\mathcal{R}^{y}=\emptyset$ and $\mathcal{I}^u_{k}=\mathcal{I}^y_{k}=\emptyset \;\forall k$.

\textbf{Disruption resources: }
The disruption capabilities in this attack scenario correspond to the ability of performing deception attacks on the actuator data channels of the local subsystem. Therefore the required resources are $\mathcal{R}_I^{u}=\{1,\dots,q_1\}$, $\mathcal{R}_I^{y}=\emptyset$, and $F=0$.

\textbf{System knowledge: }
The open-loop attack policy requires the perfect knowledge of the local zero-dynamics, denoted as $\tilde{\mathcal{K}}_{z}$ and obtained from $A_{11}$, $B_1$, $C_1$, and $A_{21}$.

\subsection{Bias Injection Attack}\label{sec:attack_bias}
Here a particular scenario of false-data injection is considered, where the adversary's goal is to inject a constant bias in the system without being detected. For this scenario, the class of $\alpha-$stealthy attacks is characterized at steady-state and a method to evaluate the corresponding impact is proposed. Furthermore, we derive the policy yielding the largest impact on the system.

\textbf{Attack policy: }
The bias injection attack is illustrated in Figure~\ref{fig:bias_attack}. The attack policy is composed of a steady-state component, the desired bias denoted as $a_\infty$, and a transient component.
For the transient, we consider that the adversary uses a linear low-pass filter so that the data corruptions are slowly converging to the steady-state values. As an example, for a set of identical first-order filters the open-loop attack sequence is described by
\begin{equation} \label{eq:bias_sequence}
a_{k+1} = \beta a_k + (1-\beta)a_\infty^*,
\end{equation}
where $a_0 = 0$ and $0<\beta<1$ can be chosen using the results from Theorem~\ref{thm:bias_transient}. The steady-state attack policy yielding the maximum impact on the physical system is described below, where the computation of $a_\infty$ is summarized in Theorem~\ref{thm:bias_attack_2} and Theorem~\ref{thm:bias_attack_infinity}.

\begin{figure}[ht]
 \def\svgwidth{6.7cm}
  \centering
  \input{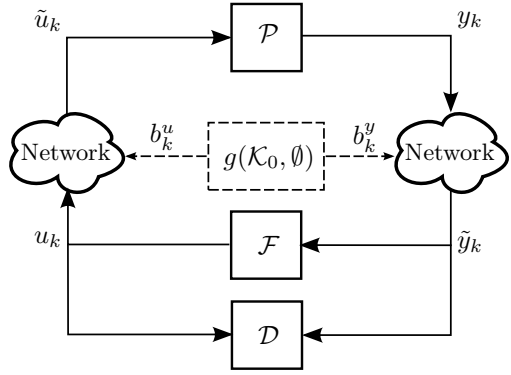}
  \caption{Schematic of the bias injection attack.\label{fig:bias_attack}}
\end{figure}

\textbf{Attack performance: }
First the steady-state policy is considered. Denote $a_{\infty}$ as the bias to be injected and recall the anomaly detector dynamics under attack~\eqref{eq:residual_dynamics_attack}. The steady-state detectability of the attack is then dependent on the steady-state value of the residual
\begin{equation*}%\label{static_eq}
r^{a}_\infty =\left ( \mathbf{C}_e (I - \mathbf{A}_e)^{-1} \mathbf{B}_e + \mathbf{D}_e \right ) a_\infty=: G_{ra} a_\infty.
\end{equation*}
The largest  $\alpha-$stealthy attacks are then characterized by
\begin{equation} \label{r gap}
\left \| G_{ra} a_\infty \right \|_2 = \delta_{\alpha}.
\end{equation}
Although attacks satisfying \eqref{r gap} could be detected during the transient, incipient attack signals slowly converging to $a_{\infty}$ may go undetected, as it is stated in Theorem~\ref{thm:bias_transient} and shown in the experiments in Section~\ref{sec:experiments}.

The impact of such attacks can be evaluated using the closed-loop dynamics under attack given by~\eqref{eq:closed_loop_attacks}. Recalling that $\eta^a_k = [x_k^{a^\top} \quad z_k^{a^\top} ]^\top$, the steady-state impact on the state is given by
\begin{equation*}
x^{a}_\infty = \left[I\quad 0\right] \left (I - \mathbf{A} \right )^{-1} \mathbf{B} a_\infty =: G_{xa} a_\infty.
\end{equation*}

Consider the following safe set defined in terms of $x^a_k$.
\begin{definition}\label{def:safe_sets_2}
 The $2-$norm safe set $\mathcal{S}_{x^a}^2$ is defined as
 \[\mathcal{S}_{x^a}^2 = \left\{x\in\mathbb{R}^n :\, \|x\|_2^2 \leq 1 \right\},\]
 and the system is said to be in a safe state if $x^a_k\in\mathcal{S}_{x^a}^2$.
\end{definition}

For the $2-$norm safe set $\mathcal{S}_{x^a}^2$, the most dangerous bias injection attack corresponds to the $\alpha-$stealthy attack yielding the largest bias in the $2-$norm sense, which can be computed by solving
\begin{equation}\label{eq:max_impact_2}
\begin{matrix}
\underset{a_\infty}{\max} \left \| G_{xa}a_\infty \right \|^2_2 \\
\\
\mbox{s.t. } \; \; \; \; \left \| G_{ra}a_\infty \right \|^2_2 \leq \delta_{\alpha}^2.
\end{matrix}
\end{equation}

\begin{lemma}~\label{lem:bounded_bias}
The optimization problem~\eqref{eq:max_impact_2} is bounded if and only if $\ker(G_{ra})\subseteq\ker(G_{xa})$.
\end{lemma}
\begin{proof}
Suppose that $\ker(G_{ra})\neq\emptyset$ and consider the subset of solutions where $a_\infty\in\ker(G_{ra})$. For this subset of solutions, the optimization problem then becomes $\max_{a_\infty\in\ker(G_{ra})} \left\| G_{xa}a_\infty \right\|^2_2$. Since the latter corresponds to a maximization of a convex function, its solution is unbounded unless $G_{xa}a_\infty = 0$ for all $a_\infty\in\ker(G_{ra})$ i.e.,  $\ker(G_{ra})\subseteq\ker(G_{xa})$. Noting that the feasible set and the objective function are bounded for all solutions $a_\infty \not \in \ker(G_{ra})$ concludes the proof.
\end{proof}

Given Lemma~\ref{lem:bounded_bias}, below we consider the non-trivial case for which it holds that $\ker(G_{ra})\subseteq\ker(G_{xa})$. The above optimization problem can be transformed into a generalized eigenvalue problem and the corresponding optimal solution characterized in terms of generalized eigenvalues and eigenvectors. Before formalizing this statement, we introduce the following result.

\begin{lemma}~\label{lem:bias_generalized}
Let $Q\in\mathbb{R}^{n\times n}$ and $P\in\mathbb{R}^{n\times n}$ be positive semi-definite matrices satisfying $\ker(Q)\subseteq\ker(P)$. Denote $\lambda^*$ as the largest generalized eigenvalue of the matrix pencil $(P,Q)$ and $v^*$ as the corresponding eigenvector. Then the matrix $P - \lambda Q$ is negative semi-definite for a generalized eigenvalue $\lambda$ if and only if $\lambda=\lambda^*$. Moreover, we have $\lambda^*\geq0$ and $x^\top(P-\lambda^* Q)x = 0$ with $Qx\neq0$ if and only if $x\in\mbox{span}(v^*)$.
\end{lemma}
\begin{proof}
The proof can be found in Appendix~\ref{app:A}.
\end{proof}

The optimal bias injection attack in the sense of~\eqref{eq:max_impact_2} is characterized by the following result.
\begin{theorem}\label{thm:bias_attack_2}
Consider the $2-$norm safe set $\mathcal{S}_{x^a}^2$ and the corresponding optimal $\alpha-$stealthy bias injection attack parameterized by the optimization problem~\eqref{eq:max_impact_2}, which is assumed to be bounded. Denote $\lambda^*$ and $v^*$ as the largest generalized eigenvalue and corresponding unit-norm eigenvector of the matrix pencil $(G_{xa}^\top G_{xa},\,G_{ra}^\top G_{ra})$. The optimal bias injection attack is given by
\begin{equation}\label{eq:optimal_bias}
 a^*_\infty = \pm\frac{\delta_{\alpha}}{\|G_{ra}v^*\|_2} v^*,
\end{equation}
and the corresponding optimal value is $\|G_{xa}a_\infty \|^2_2 = \lambda^* \delta_{\alpha}^2$. Moreover, at steady-state the system is in a safe state if and only if $\lambda^* \delta_{\alpha}^2 \leq 1$.
\end{theorem}
\begin{proof}
Let $P,\,Q\in\mathbb{R}^{n \times n}$ be positive semi-definite matrices such that $\ker(Q)\subseteq\ker(P)$. Recall that $\lambda$ is a generalized eigenvalue of $(P, Q)$ if $\mbox{rank}(P-\lambda Q ) < \mbox{normalrank}(P, Q)$, where $\mbox{normalrank}(P, Q)$ is defined as the rank of $P-\nu Q$ for almost all values of $\nu\in\mathbb{C}$. Furthermore, denote $v$ as the generalized eigenvector associated with $\lambda$ for which $(P-\lambda Q )v = 0$ with $v\not\in\ker(Q)$. The necessary and sufficient conditions for the optimization problem~\eqref{eq:max_impact_2} are given by~\cite{kn:Hiriart-Urruty2001}
\begin{equation*}
\begin{aligned}
0&=(G_{xa}^\top G_{xa}-\lambda^* G_{ra}^\top G_{ra})a_\infty^*,\\
0 &= a_\infty^{*\top} G_{ra}^\top G_{ra}a_\infty^* - \delta_\alpha^2,\\
0 &\geq y^\top (G_{xa}^\top G_{xa}-\lambda^* G_{ra}^\top G_{ra}) y ,\; \mbox{for }y\neq0.
\end{aligned}
\end{equation*}
Suppose $\lambda^*$ is the largest generalized eigenvalue of $(G_{xa}^\top G_{xa},\,G_{ra}^\top G_{ra})$ and let $v^*$ be the corresponding eigenvector. Scaling $v^*$ by $\kappa$ so that $a_\infty^* = \kappa v^*$ satisfies $\|G_{ra}a^*_\infty\|^2_2 = \delta_\alpha^2$ leads to $\kappa =\pm \frac{\delta_{\alpha}}{\|G_{ra}v^*\|_2}$, and the first and second conditions are satisfied. As for the third condition, note that $G_{xa}^\top G_{xa}-\lambda^* G_{ra}^\top G_{ra}$ is negative semi-definite by Lemma~\ref{lem:bias_generalized}, given that $\lambda^*$ is the largest generalized eigenvalue, $G_{xa}^\top G_{xa}$ and $G_{ra}^\top G_{ra}$ are positive semi-definite, and the assumption that $\ker(G_{ra})\subseteq\ker(G_{xa})$.
To conclude our proof, observe that the optimal value is given by $a_\infty^{*\top} G_{xa}^\top G_{xa} a_\infty^* = \lambda^* a_\infty^{*\top} G_{ra}^\top G_{ra}a_\infty^* = \lambda^* \delta_\alpha^2=\|x^a_\infty\|_2^2$ and thus, by definition, $x^a_\infty\in\mathcal{S}_{x^a}^2$ if and only if $\lambda^* \delta_\alpha^2 \leq 1$.
\end{proof}

More generally, the optimal bias injection attacks for ellipsoidal safe sets of the form
$\mathcal{S}_{x^a} = \left\{x^a\in\mathbb{R}^n :\, x^{a^\top}P x^a\leq 1 \right\}$,
with $P$ positive definite, can be found by replacing the objective function in~\eqref{eq:max_impact_2} by $\|P^{1/2} G_{xa}a_\infty\|_2^2$.

%\textbf{Largest infinity-norm state bias}
Similarly, consider the safe set as defined below.
\begin{definition}\label{def:safe_sets_inf}
 The infinity-norm safe set $\mathcal{S}_{x^a}^\infty$ is defined as
 \[\mathcal{S}_{x^a}^\infty = \left\{x\in\mathbb{R}^n :\, \|x\|_\infty \leq 1 \right\},\]
 and the system is said to be in a safe state if $x^a_k\in\mathcal{S}_{x^a}^\infty$.
\end{definition}

Given the infinity-norm safe set $\mathcal{S}_{x^a}^\infty$, the bias injection attack with the largest impact corresponds to the $\alpha-$stealthy attack yielding the largest bias in the infinity-norm sense. This attack can be obtained by solving the following optimization problem
\begin{equation}\label{eq:max_impact_infty}
\begin{matrix}
\underset{a_\infty}{\max} \left \| G_{xa}a_\infty \right \|_\infty \\
% \\
\mbox{s.t. } \; \; \; \; \left \| G_{ra}a_\infty \right \|_2 \leq \delta_{\alpha}.
\end{matrix}
\end{equation}

A possible method to solve this problem is to observe that \[\| G_{xa}a_\infty\|_\infty = \max_i\, \|e_i^\top G_{xa}a_\infty \|_2,\] where the vector $e_i$ is $i-$th column of the identity matrix. Thus one can transform the optimization problem~\eqref{eq:max_impact_infty} into a set of problems with the same structure as~\eqref{eq:max_impact_2}, obtaining
\begin{equation}\label{eq:bias_infinity_2}
\begin{matrix}
\underset{i}{\max}\, \underset{a^i_\infty}{\max}\, \left \| e_i^\top G_{xa}a^i_\infty \right \|_2 \\
\mbox{s.t. } \; \; \; \; \left \| G_{ra}a^i_\infty \right \|_2 \leq \delta_{\alpha}.
\end{matrix}
\end{equation}

\begin{theorem}\label{thm:bias_attack_infinity}
Consider the infinity-norm safe set $\mathcal{S}_{x^a}^\infty$ and the corresponding optimal $\alpha-$stealthy bias injection attack parameterized by the optimization problem~\eqref{eq:max_impact_infty}, which is assumed to be bounded. Let $e_i$ be the $i-$th column of the identity matrix and denote $\lambda^*_i$ and $v^*_i$ as the largest generalized eigenvalue and corresponding unit-norm eigenvector of the matrix pencil $G_{xa}^\top e_i e_i^\top G_{xa} - \lambda G_{ra}^\top G_{ra}$. Letting $\lambda^* = \max_i \lambda^*_i$, with $v^*$ as the corresponding generalized eigenvector, the optimal bias attack is given by
\begin{equation}\label{eq:optimal_bias_infinity}
 a^*_\infty = \pm\frac{\delta_{\alpha}}{\|G_{ra}v^*\|_2} v^*,
\end{equation}
and the corresponding optimal value is $\|G_{xa}a_\infty \|_\infty = \sqrt{\lambda^*} \delta_{\alpha}$. Moreover, at steady-state the system is in a safe state if and only if $\lambda^* \delta_{\alpha}^2 \leq 1$.
\end{theorem}
\begin{proof}
The proof follows directly from considering the set of optimization problems in~\eqref{eq:bias_infinity_2} and applying Theorem~\ref{thm:bias_attack_2}.
\end{proof}

Note that the steady-state value of the data corruption $a^*_\infty$ is not sufficient for the attack to be $\alpha-$stealthy, since the transients are disregarded. In practice, however, it has been observed in the fault diagnosis literature that incipient faults with slow dynamics are hard to detect~\cite{Cheng_Patton_1999}. Therefore the low-pass filter dynamics in the attack policy~\eqref{eq:bias_sequence} could be designed sufficiently slow as to difficult detection. Below we provide a method to verify whether a given filter parameter $\beta$
renders the bias attack $\alpha-$stealthy.

\begin{theorem}\label{thm:bias_transient}
Consider the attack policy $a_{k+1}=\beta a_k + (1-\beta)a^*_\infty$ with $\beta\in(0,\,1)$. The residual $r^a_k$ is characterized as the output of the autonomous system
\begin{equation}\label{eq:bias_transient}
\begin{aligned}
\psi^a_{k+1} &= \bar{A}\psi^a_k\\
r^a_k &= \bar{C}\psi^a_k
\end{aligned}
\end{equation}
with
\begin{equation*}
\begin{aligned}
\bar{A}&=\begin{bmatrix}
\mathbf{A}_e & \mathbf{B}_e & 0\\
0 & \beta I  & (1-\beta)I\\
0 & 0 & I
\end{bmatrix},\quad \psi^a_0=\begin{bmatrix}
0\\
0\\
a^*_{\infty}
\end{bmatrix},
\\
\bar{C} &= \begin{bmatrix}
\mathbf{C}_e & \mathbf{D}_e & 0
\end{bmatrix}.
\end{aligned}
\end{equation*}

Moreover, the attack policy is $\alpha-$stealthy for a given $\beta$ if the following optimization problem admits a solution
\begin{equation}
\begin{array}{rr}
& \underset{\gamma,P}{\min}\quad \gamma\quad  \\
& \mbox{s.t. }  \quad \gamma \leq \delta_\alpha^2,\\
 & P \succ 0,\\
& \psi_0^{a^\top}
P
\psi^a_0 \leq 1, \\
 & \begin{bmatrix}
P               & \bar{C}^\top\\
\bar{C}    &  \gamma I
\end{bmatrix}
\succeq 0,\\
 & \bar{A}^\top
P
\bar{A}
-
P \prec 0.
\end{array}
\end{equation}
\end{theorem}
\begin{proof}
 The autonomous system is directly obtained by considering the augmented state $\psi^a=[\xi_{k|k}^{a^\top}\; l_k^\top\; s_k^\top]^\top$, where $l_k$ is the state of the low-pass filter bank and $s_k$ the integral state initialized at $s_0 = a_\infty$. Given this autonomous system, one observes that the attack is $\alpha-$stealthy if and only if the corresponding output-peak $\|r^a_k\|_2^2$ is bounded by $\delta_\alpha^2$ for all $k\geq0$, given the initial condition parameterized by $\alpha^*_\infty$. The remainder of the proof follows directly from the results in~\cite{Boyd_LMI94} regarding output-peak bounds for autonomous systems.
\end{proof}

However, the output-peak bounds are in general conservative and thus the conditions in the previous theorem are only sufficient.

\textbf{Disclosure resources: }
Similarly to the zero attack, no disclosure capabilities are required for this attack, since the attack policy is open-loop. Therefore we have $\mathcal{R}^{u}=\mathcal{R}^{y}=\emptyset$ and $\mathcal{I}^u_{k}=\mathcal{I}^u_{k}=\emptyset \,\forall k$.

\textbf{Disruption resources: }
The biases may be added to both the actuator and sensor data, hence the required resources are $\mathcal{R}_I^{u} \subseteq \{1,\dots,q\}$, $\mathcal{R}_I^{y}\subseteq \{1,\dots,p\}$. Since no physical attack is performed, we have $F=0$.

\textbf{System knowledge: }
As seen in~\eqref{eq:max_impact_2}, the open-loop attack policy~\eqref{eq:bias_sequence} requires the knowledge of the closed-loop system and anomaly detector steady-state gains $G_{ra}$ and $G_{xa}$, which we denoted as $\mathcal{K}_0$ as shown in Figure~\ref{fig:bias_attack}.

\section{Experiments}\label{sec:experiments}
In this section we present our testbed and report experiments on staged cyber attacks following the different scenarios described in the previous section.

\subsection{Quadruple-Tank Process}

Our testbed consists of a Quadruple-Tank Process (QTP) \cite{Johansson2000} controlled through a wireless communication network, as shown in Figure~\ref{figQTP}.
\begin{figure}[thpb]
   \centering
   \includegraphics[scale=0.9]{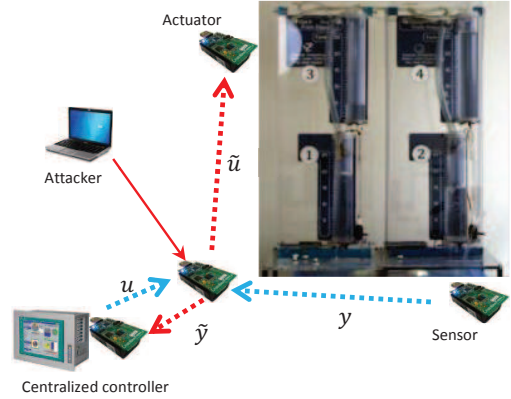}
   \caption{Schematic diagram of the testbed with the Quadruple-Tank Process and a multi-hop communication network.}
\label{figQTP}
\end{figure}

The plant model can be found in~\cite{Johansson2000}
\begin{equation}\label{eqQTP}
   \begin{split}
       \dot{h}_1 &= -\frac{a_1}{A_1}\sqrt{2gh_1}+\frac{a_3}{A_1}\sqrt{2gh_3}+\frac{\gamma_1 k_1}{A_1}u_1,\\
      \dot{h}_2 &= -\frac{a_2}{A_2}\sqrt{2gh_2}+\frac{a_4}{A_2}\sqrt{2gh_4}+\frac{\gamma_2 k_2}{A_2}u_2,\\
      \dot{h}_3 &= -\frac{a_3}{A_3}\sqrt{2gh_3}+\frac{(1-\gamma_2) k_2}{A_3}u_2,\\
      \dot{h}_4 &= -\frac{a_4}{A_4}\sqrt{2gh_4}+\frac{(1-\gamma_1) k_1}{A_4}u_1,
   \end{split}
\end{equation}
where $h_i \in[0,\; 30]$ are the heights of water in each tank, $A_i$ the cross-section area of the tanks, $a_i$ the cross-section area of the outlet hole, $k_i$ the pump constants, $\gamma_i$ the flow ratios and $g$ the gravity acceleration. The nonlinear plant model is linearized for a given operating point. Moreover, given the range of the water levels, the following safe set is considered $\mathcal{S}_x = \{x\in\mathbb{R}^n:\; \|x-\sigma\mathbf{1} \|_\infty\leq 15,\, \sigma=15\}$, where $\mathbf{1}\in\mathbb{R}^n$ is a vector with all entries set to $1$.

The QTP is controlled using a centralized LQG controller with integral action running in a remote computer and a wireless network is used for the communications. A Kalman-filter-based anomaly detector is also running in the remote computer and alarms are triggered according to~\eqref{eq:residue_threshold}, for which we computed $\delta_r=0.15$ and chose $\delta_{\alpha}=0.25$ for illustration purposes.
The communication network is multi-hop, having one additional wireless device relaying the data, as illustrated in Figure~\ref{figQTP}.

\subsection{Denial-of-Service Attack}
Here we consider the case where the QTP suffers a DoS attack on both sensors, while operating at a constant set-point. The state and residual trajectories from this experiment are presented in Figure~\ref{fig_dos_graph}. The DoS attack follows a Bernoulli model~\cite{AminCardenasSastry-HSCC-2009} with $p=0.9$ as the probability of packet loss and the last received data is used in the absence of data. From Proposition~\ref{thm:DoS_stability}, we have that the closed-loop system under such DoS attack is exponentially stable.

\begin{figure}[t]
   \centering
   \includegraphics[width=0.45\hsize]{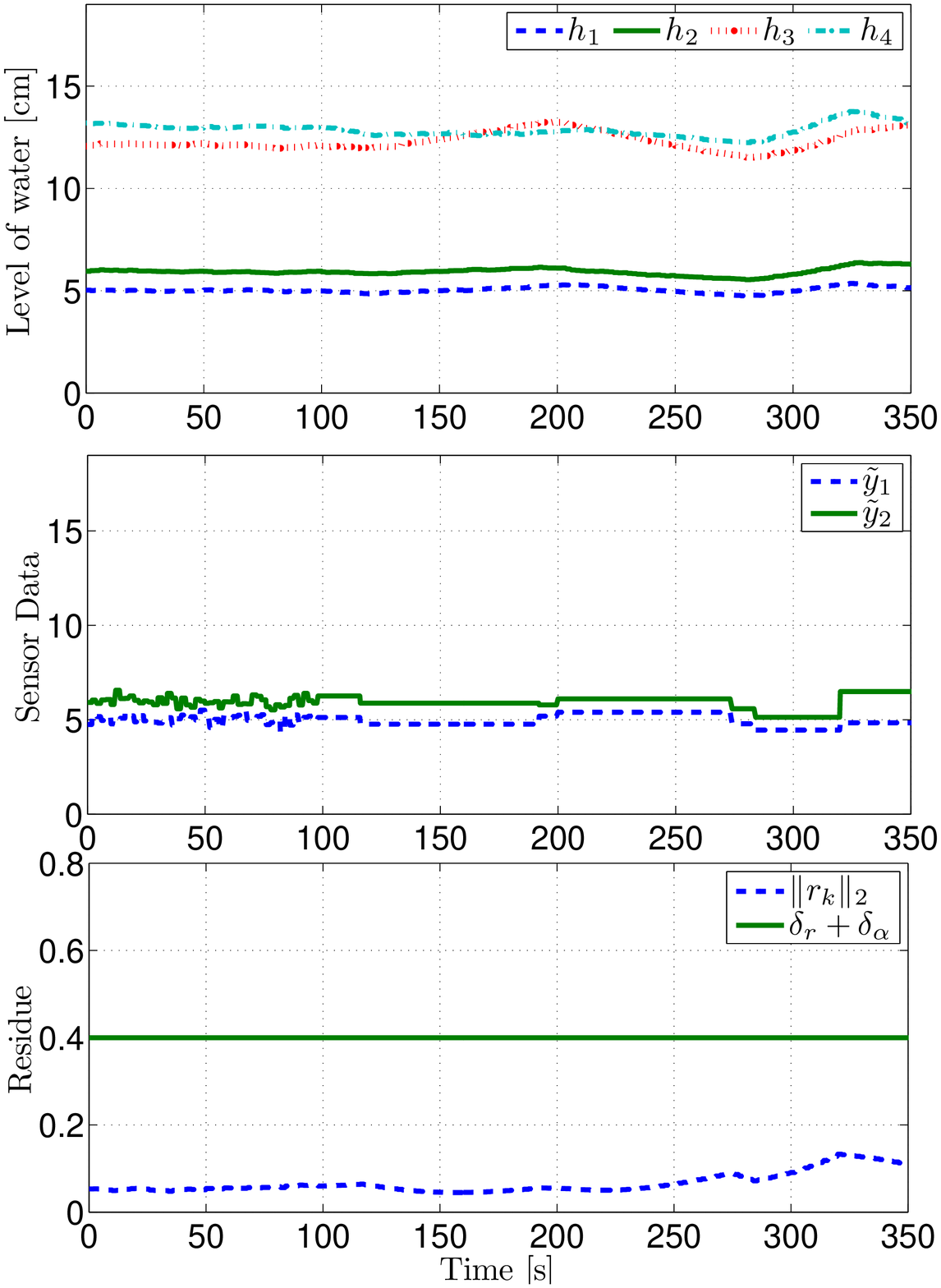} %[scale=0.35]
   \caption{Results for the DoS attack performed against both sensors since $t\approx 100\,s$.\label{fig_dos_graph}}
\end{figure}

The DoS attack initiates at $t\approx 100\,s$, leading to an increase in the residual due to successive packet losses. However the residual remained below the threshold during the attack and there were no significant changes in the system's state.

\subsection{Replay Attack}
In this scenario, the QTP is operating at a constant set-point while a hacker desires to steal water from tank 4, the upper tank on the right side.
An example of this attack is presented in Figure~\ref{fig_replay_graph}, where the replay attack policy is the one described in Section~\ref{sec:replay_attack}. The adversary starts  by replaying past data from $y_2$ at $t\approx 90\,s$ and then begins stealing water from tank 4 at $t\approx 100\,s$. Tank 4 is successfully emptied and the attacks stops removing water at $t\approx180\,s$. To ensure stealthiness, the replay attack continues until the system recovered its original setpoint at $t\approx 280\,s$. As we can see, the residue stays below the alarm threshold and therefore the attack is not detected.

\subsection{Zero-Dynamics Attack}
The QTP has a non-minimum phase configuration in which the plant possesses an unstable zero. In this case, as discussed in Section~\ref{sec:attack_zero}, an adversary able to corrupt all the actuator channels may launch a false-data injection attack where the false-data follows the zero-dynamics. Moreover, since the safe region is described by the set $\mathcal{S}_x = \{x\in\mathbb{R}^n:\; \|x-\sigma\mathbf{1} \|_\infty\leq 15,\, \sigma=15\}$, from Theorem~\ref{thm:zero_unsafe} we expect that the zero-dynamics attack associated with the unstable zero can drive the system to an unsafe region. This scenario is illustrated in Figure~\ref{fig_zero_graph}.

\begin{figure}[t]
   \centering
   \includegraphics[width=0.5\hsize]{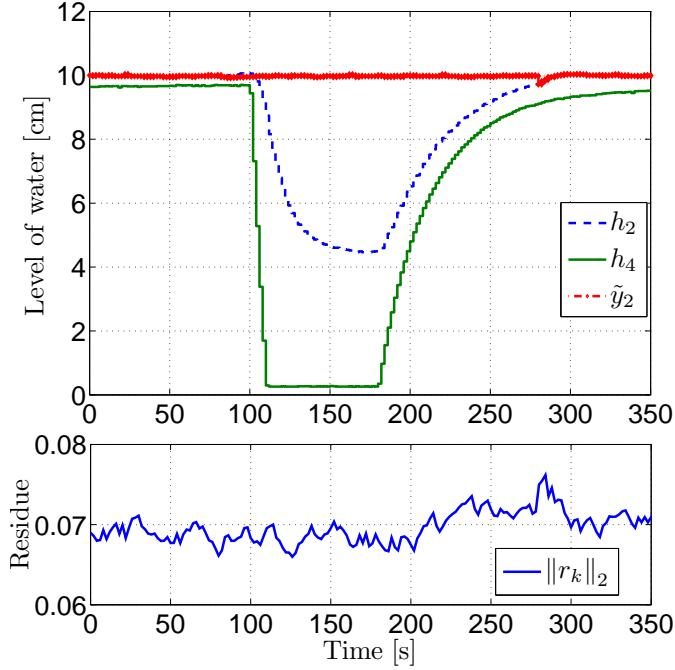}% [scale=0.35]
   \caption{Results for the replay attack performed against sensor 2 from $t\approx 90\,s$ to $t\approx 280\,s$. Additionally, the adversary opens the tap of tank 4 at $t \approx 100\,s$ and closes it at $t \approx 180\,s$.\label{fig_replay_graph}}
\end{figure}

\begin{figure}[t]
   \centering
   \includegraphics[width=0.5\hsize]{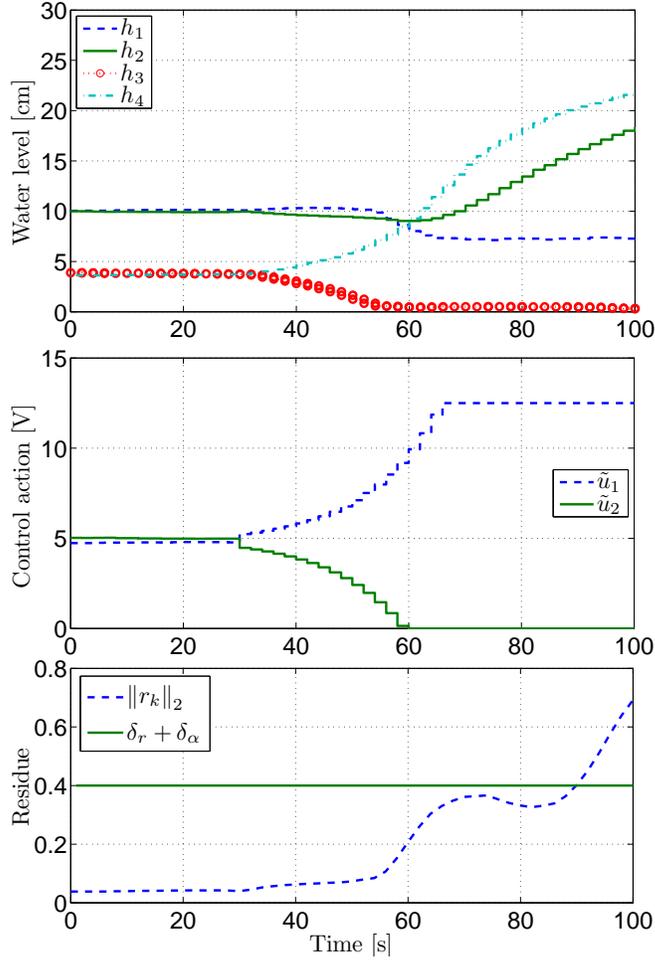} %[scale=0.35]
    \caption{Results for the zero-dynamics attack starting at $t \approx 30\,s$. Tank 3 is emptied at $t\approx 55\,s$, resulting in a steep increase in the residual since the linearized model is no longer valid.\label{fig_zero_graph}}
\end{figure}

The adversary's goal is to either empty or overflow at least one of the tanks, considered as an unsafe state. The attack on both actuators begins at $t\approx30\,s$, causing a slight increase in the residual. Tank 3 becomes empty at $t\approx55\,s$ and shortly after actuator 2 saturates, producing a steep increase in the residual which then crosses the threshold. However, note that the residual was below the threshold when the unsafe state was reached.

After saturation of the water level and the actuators, the system dynamics change and therefore the attack signal no longer corresponds to the zero-dynamics and is detected, although it has already damaged the system. Thus these attacks are particularly dangerous in processes that have unstable zero-dynamics and in which the actuators are over-dimensioned, allowing the adversary to perform longer attacks before saturating.

\subsection{Bias Injection Attack}

The results for the case where $u_1$ and $y_1$ are respectively corrupted with $b^u_\infty$ and $b^y_\infty$ are presented in the Figure~\ref{figSA}. In this scenario, the adversary aimed at driving the system out of the safe set $\mathcal{S}_x$ while remaining stealthy for $\delta_\alpha = 0.25$. The bias was slowly injected using a first-order low-pass filter with $\beta=0.95$ and the following steady-state value, computed using Theorem~\ref{thm:bias_attack_infinity}, $a_\infty = [b^u_\infty\, b^y_\infty]^\top = [2.15\, -9.42]^\top$.

\begin{figure}[t]
   \centering
   \includegraphics[width=0.5\hsize]{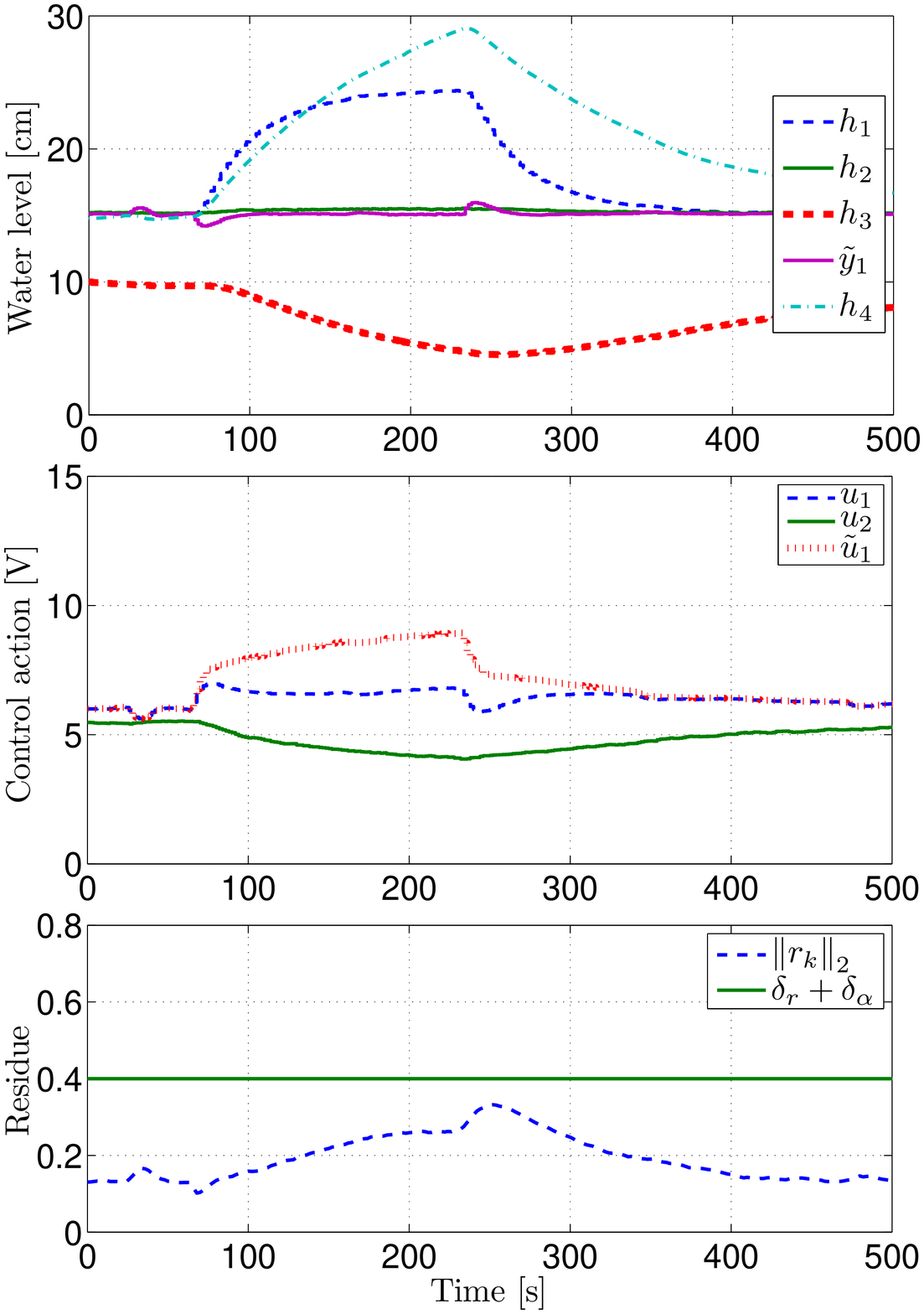} %[scale=0.35]
   \caption{Results for the bias attack against the actuator 1 and sensor 1 in the minimum phase QTP. The attack is launched using a low-pass filter in the instant $t \approx 70\,s$ and stopped at $t \approx 230\,s$.}
\label{figSA}
\end{figure}

The bias injection began at $t\approx70\,s$ and led to an overflow in tank 4 at $t\approx225\,s$. At that point, the adversary started removing the bias and the system recovered the original setpoint at $t\approx 350\,s$. The residual remained within the allowable bounds throughout the attack, thus the attack was not detected.

\section{Conclusions}\label{sec:conc}
In this paper we have analyzed the security of networked control systems. A novel attack space based on the adversary's system knowledge, disclosure, and disruption resources was proposed and the corresponding adversary model described. Attack scenarios corresponding to replay, zero-dynamics, and bias injection attacks were analyzed using this framework. In particular the maximum impact of stealthy bias injection attacks was derived and it was shown that the corresponding policy does not require perfect model knowledge. These attack scenarios were illustrated using an experimental setup based on a quadruple-tank process controlled over a wireless network.

\section{Acknowledgments}
This work was supported in part by the European Commission through the HYCON2 project, the EIT-ICT Labs through the project SESSec-EU, the Swedish Research Council under Grants 2007-6350 and 2009-4565, and the Knut and Alice Wallenberg Foundation.

\bibliographystyle{plain}
\bibliography{references_all}

\appendix
\section{Proof of Lemma~\ref{lem:bias_generalized}}\label{app:A}
%\begin{lemma*}{\ref{lem:bias_generalized}}
%Let $Q\in\mathbb{R}^{n\times n}$ and $P\in\mathbb{R}^{n\times n}$ be positive semi-definite matrices satisfying $\ker(Q)\subseteq\ker(P)$. Denote $\lambda^*$ as the largest generalized eigenvalue of the matrix pencil $(P,Q)$ and $v^*$ as the corresponding eigenvector. Then the matrix $P - \lambda Q$ is negative semi-definite for a generalized eigenvalue $\lambda$ if and only if $\lambda=\lambda^*$. Moreover, we have $\lambda^*\geq0$ and $x^\top(P-\lambda^* Q)x = 0$ with $Qx\neq0$ if and only if $x\in\mbox{span}(v^*)$.
%\end{lemma*}
%\begin{proof}
Recall that $\lambda$ is a generalized eigenvalue of $(P, Q)$ if $\mbox{rank}(P-\lambda Q ) < \mbox{normalrank}(P, Q)$, where $\mbox{normalrank}(P, Q)$ is defined as the rank of $P-\nu Q$ for almost all values of $\nu\in\mathbb{C}$. Furthermore, denote $v$ as the generalized eigenvector associated with $\lambda$ for which $(P-\lambda Q )v = 0$ with $v\not\in\ker(Q)$.

Define $T=[V_{\bar{N}} \, V_N] \in \mathbb{R}^{n\times n}$ where the columns of $V_N$ are a basis for $\ker(Q)$ and $V_{\bar{N}}$ is chosen such that $T$ is nonsingular. Given that $\ker(Q)\subseteq\ker(P)$, the coordinate transformation induced by $T$ leads to
\[T(P-\lambda Q)T^{-1} = \begin{bmatrix}\tilde{P}-\lambda \tilde{Q} & 0\\ 0 & 0\end{bmatrix},\]
where $\tilde{Q}\succ 0$ and $\tilde{P}\succeq 0$ and we conclude that $P-\lambda Q \preceq0$ if and only if $\tilde{P}-\lambda \tilde{Q} \preceq0$. Additionally, we see that all the non-zero generalized eigenvalues of $(P,Q)$ need to reduce the rank of $\tilde{P}-\lambda \tilde{Q}$ and thus need to be positive. Hence we have proved that all generalized eigenvalues are non-negative and that $\lambda^*\geq0$.

Now we show that $\tilde{P}-\lambda \tilde{Q}$ is indefinite for all generalized eigenvalues $0 < \lambda < \lambda^*$. Let $\bar\lambda > 0$ be a generalized eigenvalue of $(\tilde{P},\tilde{Q})$ with the associated eigenvector $\bar{v}$. Then $\bar{v}^{\top}(\tilde{P}-\lambda \tilde{Q})\bar{v} = (\bar{\lambda} - \lambda)\bar{v}^{\top}\tilde{Q}\bar{v}$, which can be made positive or negative for all generalized eigenvalues $\lambda \in (0,\, \lambda^*)$ and thus our assertion is proved.

As the next step, we show that $\tilde{P}-\lambda^* \tilde{Q}\preceq 0$. Since $\tilde{Q}$ is invertible, the generalized eigenvalues of $(\tilde{P},\tilde{Q})$ correspond to the eigenvalues of the positive semi-definite matrix $M\tilde{P}M$ with $M=\tilde{Q}^{-1/2}$. Furthermore note that $\tilde{P}-\lambda^* \tilde{Q}\preceq 0$ is equivalent to having $M\tilde{P}M-\lambda^* I\preceq 0$, which holds since $M\tilde{P}M$ is positive semi-definite with $\lambda^*$ as the largest eigenvalue.

All that is left to show now is that $x^\top(P-\lambda^* Q)x = 0$ with $Qx\neq0$ if and only if $x\in\mbox{span}(v^*)$. Given the condition $Qx\neq0$, it is enough to verify that $x^\top(\tilde{P}-\lambda^* \tilde{Q})x = 0$ for $x\neq0$ if and only if $x\in\mbox{span}(\tilde{v}^*)$, where $\tilde{v}^*$ is the generalized eigenvector of $(\tilde{P},\tilde{Q})$ associated with $\lambda^*$. The proof is concluded by recalling that $\tilde{P}-\lambda^* \tilde{Q}\preceq 0$, hence $x^\top(\tilde{P}-\lambda^* \tilde{Q})x = 0$ if and only if $x$ belongs to the subspace spanned by the eigenvectors associated with $\lambda^*$.
%\end{proof}

\end{document}